\documentclass{amsart}
\usepackage{amsmath}
\usepackage{amsthm}
\usepackage{amssymb}

\usepackage{tikz}
\usetikzlibrary{arrows}

\usepackage{clontzDefinitions}

\renewcommand{\vec}{\mathbf}

      \theoremstyle{plain}
      \newtheorem{theorem}{Theorem}
      \newtheorem{lemma}[theorem]{Lemma}
      \newtheorem{corollary}[theorem]{Corollary}
      \newtheorem{proposition}[theorem]{Proposition}
      
      \newtheorem{question}[theorem]{Question}

      \theoremstyle{definition}
      \newtheorem{definition}[theorem]{Definition}

      \theoremstyle{remark}

      \theoremstyle{plain}
      \newtheorem*{theorem*}{Theorem}
      \newtheorem*{lemma*}{Lemma}
      \newtheorem*{corollary*}{Corollary}
      \newtheorem*{proposition*}{Proposition}
      \newtheorem*{conjecture*}{Conjecture}
      \newtheorem*{question*}{Question}
      \newtheorem*{claim*}{Claim}

      \theoremstyle{definition}
      \newtheorem*{definition*}{Definition}
      \newtheorem*{example*}{Example}
      \newtheorem*{game*}{Game}

      \theoremstyle{remark}
      \newtheorem*{remark*}{Remark}

\usepackage{lineno}
\begin{document}

\title{Arhangelskii's \(\alpha\)-principles and selection games}

\author{Steven Clontz}
\address{Department of Mathematics and Statistics,
The University of South Alabama,
Mobile, AL 36688}
\email{sclontz@southalabama.edu}


%\subjclass[2010]{54D20, 54D45, 91A44}

\begin{abstract}
Arhangelskii's properties \(\alpha_2\) and \(\alpha_4\)
defined for convergent sequences
may be characterized in terms of Scheeper's selection principles.
We generalize these results to hold for more general collections
and consider these results in terms of selection games.
\end{abstract}

\maketitle

The following characterizations were given as Definition 1 by Kocinac in
\cite{MR2417134}.

\begin{definition}
\term{Arhangelskii's \(\alpha\)-principles} \(\alpha_i(\mc A,\mc B)\) are defined as follows
for \(i\in\{1,2,3,4\}\).
Let \(A_n\in\mc A\) for all \(n<\omega\); then there exists \(B\in\mc B\) such that:
\begin{itemize}
\item[\(\alpha_1\):] \(A_n\cap B\) is cofinite in \(A_n\) for all \(n<\omega\).
\item[\(\alpha_2\):] \(A_n\cap B\) is infinite for all \(n<\omega\).
\item[\(\alpha_3\):] \(A_n\cap B\) is infinite for infinitely-many \(n<\omega\).
\item[\(\alpha_4\):] \(A_n\cap B\) is non-empty for infinitely-many \(n<\omega\).
\end{itemize}
\end{definition}

When \((\mc A,\mc B)\) is omitted, it is assumed that \(\mc A=\mc B\) is the
collection \(\Gamma_{X,x}\) of sequences converging to some point \(x\in X\),
as introduced by Arhangelskii in \cite{MR0394575}.
Provided \(\mc A\) only contains infinite sets, it's easy to see that
\(\alpha_n(\mc A,\mc B)\) implies \(\alpha_{n+1}(\mc A,\mc B)\).

We aim to relate these to the following games.

\begin{definition}
  The \term{selection game} \(G_1(\mc A,\mc B)\) (resp. \(G_{fin}(\mc A,\mc B)\))
  is an \(\omega\)-length game involving Players \(\plI\) and \(\plII\). 
  During round \(n\), \(\plI\) chooses
  \(A_n\in\mc A\), followed by \(\plII\) choosing \(a_n\in A_n\)
  (resp. \(F_n\in[A_n]^{<\aleph_0}\)).
  Player \(\plII\) wins in the case that \(\{a_n:n<\omega\}\in\mc B\)
  (resp. \(\bigcup\{F_n:n<\omega\}\in\mc B\)),
  and Player \(\plI\) wins otherwise.
\end{definition}

Such games are well-represented in the literature; see \cite{MR1378387}
for example.
We will also consider the similarly-defined games \(G_{<2}(\mc A,\mc B)\)
(\(\plII\) chooses 0 or 1 points from each choice by \(\plI\)) and \(G_{cf}(\mc A,\mc B)\)
(\(\plII\) chooses cofinitely-many points).

\begin{definition}
  Let \(P\) be a player in a game \(G\). \(P\) has a \term{winning strategy}
  for \(G\), denoted \(P\win G\), if \(P\) has a strategy that defeats every
  possible counterplay by their opponent. If a strategy only relies on the
  round number and ignores the moves of the opponent, the strategy is said
  to be \term{predetermined}; the existence of a predetermined winning strategy
  is denoted \(P\prewin G\).
\end{definition}

We briefly note that the statement \(\plI\notprewin G_\star(\mc A,\mc B)\)
is often denoted as the \term{selection principle} \(S_\star(\mc A,\mc B)\).

\begin{definition}
Let \(\Gamma_{X,x}\) be the collection of non-trivial sequences \(S\subseteq X\) converging to \(x\),
that is, infinite subsets of \(X\setminus\{x\}\) such that for each neighborhood \(U\) of \(x\),
\(S\cap U\) is cofinite in \(S\).
\end{definition}

\begin{definition}
Let \(\Gamma_X\) be the collection of open \term{\(\gamma\)-covers} \(\mc U\) of \(X\),
that is, infinite open covers of \(X\) such that \(X\not\in \mc U\) and for each \(x\in X\),
\(\{U\in\mc U:x\in U\}\) is cofinite in \(\mc U\).
\end{definition}

The similarity in nomenclature follows from the observation that every non-trivial sequence in
\(C_p(X)\) converging to the zero function \(\vec{0}\) naturally defines a corresponding
\(\gamma\)-cover in \(X\), see e.g. Theorem 4 of
\cite{MR1396994}.

The equivalence of \(\alpha_2(\Gamma_{X,x}\Gamma_{X,x})\) and 
\(\plI\notprewin G_1(\Gamma_{X,x},\Gamma_{X,x})\) was briefly asserted by Sakai
in the introduction of \cite{MR2280899}; the similar
equivalence of \(\alpha_4(\Gamma_{X,x}\Gamma_{X,x})\) and 
\(\plI\notprewin G_{fin}(\Gamma_{X,x},\Gamma_{X,x})\) seems to be folklore.
In fact, these relationships hold in more generality.

Note that by these definitions, convergent sequences (resp. \(\gamma\)-covers) may be uncountable,
but any infinite subset of either would remain a convergent sequence (resp. \(\gamma\)-cover),
in particular, countably infinite subsets. We capture this idea as follows.

\begin{definition}
Say a collection \(\mc A\) is \term{\(\Gamma\)-like} if it satisfies the following
for each \(A\in\mc A\).
\begin{itemize}
\item \(|A|\geq\aleph_0\).
\item If \(A'\subseteq A\) and \(|A'|\geq\aleph_0\), then \(A'\in\mc A\).
\end{itemize}
\end{definition}

We also require the following.

\begin{definition}
Say a collection \(\mc A\) is \term{almost-\(\Gamma\)-like} if
for each \(A\in\mc A\), there is \(A'\subseteq A\) such that:
\begin{itemize}
\item \(|A'|=\aleph_0\).
\item If \(A''\) is a cofinite subset of \(A'\), then \(A''\in\mc A\).
\end{itemize}
\end{definition}

So all \(\Gamma\)-like sets are almost-\(\Gamma\)-like.

We are now able to prove a few general equivalences between \(\alpha\)-princples
and selection games.

\section{On \(\alpha_2(\mc A,\mc B)\) and \(G_1(\mc A,\mc B)\)}

\begin{theorem}
Let \(\mc A\) be almost-\(\Gamma\)-like and \(\mc B\) be \(\Gamma\)-like. 
Then \(\alpha_2(\mc A,\mc B)\) holds if and only
if \(\plI\notprewin G_1(\mc A,\mc B)\).
\end{theorem}

\begin{proof}
We first assume \(\alpha_2(\mc A,\mc B)\) and let \(A_n\in\mc A\) for \(n<\omega\)
define a predetermined strategy for \(\plI\).
We may apply \(\alpha_2(\mc A,\mc B)\) to choose \(B\in\mc B\) such that
\(|A_n\cap B|\geq\aleph_0\). We may then choose \(a_n\in(A_n\cap B)\setminus\{a_i:i<n\}\)
for each \(n<\omega\). It follows that \(B'=\{a_n:n<\omega\}\in\mc B\) since
\(B'\) is an infinite subset of \(B\in\mc B\); therefore \(A_n\) does not define
a winning predetermined strategy for \(\plI\).

Now suppose \(\plI\notprewin G_1(\mc A,\mc B)\). Given \(A_n\in\mc A\) for \(n<\omega\),
first choose \(A_n'\in\mc A\) such that \(A_n'=\{a_{n,j}:j<\omega\}\subseteq A_n\),
\(j<k\) implies \(a_{n,j}\not=a_{n,k}\),
and \(A_{n,m}=\{a_{n,j}:m\leq j<\omega\}\in\mc A\).
Finally choose
some \(\theta:\omega\to\omega\) such that \(|\theta^{\leftarrow}(n)|=\aleph_0\) for
each \(n<\omega\).

Since playing \(A_{\theta(m),m}\) during round \(m\)
does not define a winning strategy for \(\plI\) in
\(G_1(\mc A,\mc B)\), \(\plII\) may choose \(x_m\in A_{\theta(m),m}\)
such that \(B=\{x_m:m<\omega\}\in\mc B\).
Choose \(i_m<\omega\) for each \(m<\omega\) such that
\(x_m=a_{\theta(m),i_m}\), noting \(i_m\geq m\).
It follows that 
\(A_n\cap B\supseteq\{a_{\theta(m),i_m}:m\in\theta^{\leftarrow}(n)\}\).
Since for each \(m\in\theta^{\leftarrow}(n)\) there exists
\(M\in\theta^{\leftarrow}(n)\) such that \(m\leq i_m<M\leq i_{M}\),
and therefore \(a_{\theta(m),i_m}\not=a_{\theta(m),i_{M}}=a_{\theta(M),i_{M}}\),
we have shown that \(A_n\cap B\) is infinite. Thus \(B\) witnesses
\(\alpha_2(\mc A,\mc B)\).
\end{proof}

While \(\alpha_2(\mc A,\mc B)\) involves infinite intersection and
\(G_1(\mc A,\mc B)\) involves single selections, the previous result is made
more intuitive given the following result, shown for \(\mc A=\mc B=\Gamma_{X,x}\)
by Nogura in \cite{MR812643}.

\begin{definition}
\(\alpha_2'(\mc A,\mc B)\) is the following claim:
if \(A_n\in\mc A\) for all \(n<\omega\), then there exists \(B\in\mc B\) such that
\(A_n\cap B\) is nonempty for all \(n<\omega\).
\end{definition}

(Note that \(\alpha_5\) is sometimes used in the literature in place of \(\alpha_2'\).)

\begin{proposition}
If \(\mc A\) is almost-\(\Gamma\)-like, then
\(\alpha_2(\mc A,\mc B)\) is equivalent to \(\alpha_2'(\mc A,\mc B)\).
\end{proposition}

\begin{proof}
The forward implication is immediate, so we assume \(\alpha_2'(\mc A,\mc B)\).
Given \(A_n\in\mc A\), we apply the almost-\(\Gamma\)-like property to obtain
\(A_n'=\{a_{n,m}:m<\omega\}\subseteq A_n\) such that
\(A_{n,m}=A_n\setminus\{a_{i,j}:i,j<m\}\in\mc A\) for all \(m<\omega\).

By applying \(\alpha_2'(\mc A,\mc B)\) to \(A_{n,m}\), we obtain
\(B\in\mc B\) such that \(A_{n,m}\cap B\) is nonempty for all \(n,m<\omega\).
Since it follows that \(A_n\cap B\) is infinite for all \(n<\omega\),
we have established \(\alpha_2(\mc A,\mc B)\).
\end{proof}

\section{On \(\alpha_4(\mc A,\mc B)\) and \(G_{fin}(\mc A,\mc B)\)}

A similar correspondence exists between \(\alpha_4(\mc A,\mc B)\)
and \(G_{fin}(\mc A,\mc B)\).

\begin{theorem}
Let \(\mc A\) be almost-\(\Gamma\)-like and \(\mc B\) be \(\Gamma\)-like. 
Then \(\alpha_4(\mc A,\mc B)\) holds if and only
if \(\plI\notprewin G_{<2}(\mc A,\mc B)\) if and only if
\(\plI\notprewin G_{fin}(\mc A,\mc B)\).
\end{theorem}

\begin{proof}
We first assume \(\alpha_4(\mc A,\mc B)\) and let \(A_n\in\mc A\) for \(n<\omega\)
define a predetermined strategy for \(\plI\) in 
\(G_{<2}(\mc A,\mc B)\). We then may choose \(A_n'\in\mc A\) where
\(A_n'=\{a_{n,j}:j<\omega\}\subseteq A_n\), \(j<k\) implies
\(a_{n,j}\not=a_{n,k}\), and \(A_n''=A_n'\setminus\{a_{i,j}:i,j<n\}\in\mc A\).

By applying \(\alpha_4(\mc A,\mc B)\) to \(A_n''\), we obtain \(B\in\mc B\)
such that \(A_n''\cap B\not=\emptyset\) for infintely-many \(n<\omega\).
We then let \(F_n=\emptyset\) when \(A_n''\cap B=\emptyset\), and
\(F_n=\{x_n\}\) for some \(x_n\in A_n''\cap B\) otherwise. Then we will have that
\(B'=\bigcup\{F_n:n<\omega\}\subseteq B\) belongs to \(\mc B\) once we show that
\(B'\) is infinite. To see this, for \(m\leq n<\omega\) note that either \(F_m\) is
empty (and we let \(j_m=0\)) or \(F_m=\{a_{m,j_m}\}\)
for some \(j_m\geq m\); choose \(N<\omega\) such that \(j_m<N\) for all
\(m\leq n\) and \(F_N=\{x_N\}\). Thus \(F_m\not=F_N\) for all \(m\leq n\) since
\(x_{N}\not\in\{a_{i,j}:i,j< N\}\). Thus \(\plII\) may defeat the predetermined
strategy \(A_n\) by playing \(F_n\) each round.

Since \(\plI\notprewin G_{<2}(\mc A,\mc B)\) immediately implies
\(\plI\notprewin G_{fin}(\mc A,\mc B)\), we assume the latter. Given \(A_n\in\mc A\)
for \(n<\omega\), we note this defines a (non-winning) predetermined 
strategy for \(\plI\), so \(\plII\) may choose \(F_n\in[A_n]^{<\aleph_0}\) such that
\(B=\bigcup\{F_n:n<\omega\}\in\mc B\). Since \(B\) is infinite, we note
\(F_n\not=\emptyset\) for infinitely-many \(n<\omega\). Thus \(B\) witnesses
\(\alpha_4(\mc A,\mc B)\) since \(A_n\cap B\supseteq F_n\not=\emptyset\) for
infinitely-many \(n<\omega\).
\end{proof}

This shows that \(\plII\) gains no advantage from picking more than one
point per round. This in fact only depends on \(\mc B\) being
\(\Gamma\)-like, which we formalize in the following results.

\begin{theorem}
Let \(\mc B\) be \(\Gamma\)-like. Then \(\plI\prewin G_{<2}(\mc A,\mc B)\)
if and only if \(\plI\prewin G_{fin}(\mc A,\mc B)\).
\end{theorem}

\begin{proof}
Assume \(\bigcup\mc A\) is well-ordered.
Given a winning predetermined strategy \(A_n\) for \(\plI\) in
\(G_{<2}(\mc A,\mc B)\), consider \(F_n\in[A_n]^{<\aleph_0}\). We set
\[
  F_n^*
=
  \begin{cases}
    \emptyset
      & \text{ if }
    F_n\setminus\bigcup\{F_m:m<n\}=\emptyset
      \\
    \{\min(
      F_n\setminus\bigcup\{F_m:m<n\}
    )\}
      & \text{ otherwise}
  \end{cases}
\]
Since \(|F_n^*|<2\), we have that \(\bigcup\{F_n^*:n<\omega\}\not\in\mc B\).
In the case that \(\bigcup\{F_n^*:n<\omega\}\) is finite, we immediately
see that \(\bigcup\{F_n:n<\omega\}\) is also finite and therefore not
in \(\mc B\). Otherwise \(\bigcup\{F_n^*:n<\omega\}\not\in\mc B\)
is an infinite subset of \(\bigcup\{F_n:n<\omega\}\), and thus
\(\bigcup\{F_n:n<\omega\}\not\in\mc B\) too. Therefore
\(A_n\) is a winning predetermined strategy for \(\plI\) in
\(G_{fin}(\mc A,\mc B)\) as well.
\end{proof}

\begin{theorem}
Let \(\mc B\) be \(\Gamma\)-like. Then \(\plI\win G_{<2}(\mc A,\mc B)\)
if and only if \(\plI\win G_{fin}(\mc A,\mc B)\).
\end{theorem}

\begin{proof}
Assume \(\bigcup\mc A\) is well-ordered.
Suppose \(\plI\win G_{<2}(\mc A,\mc B)\) is witnessed by the strategy
\(\sigma\). Let \(\tuple{}^\star=\tuple{}\), and for 
\(s\concat\tuple{F}\in([\bigcup\mc A]^{<\aleph_0})^{<\omega}
\setminus\{\tuple{}\}\) let 
\[
  (s\concat\tuple{F})^\star
=
  \begin{cases}
    s^\star\concat\tuple{\emptyset}
      & \text{ if }
    F\setminus\bigcup\ran{s}=\emptyset
      \\
    s^\star\concat\tuple{
      \{\min(F\setminus\bigcup\ran{s})\}
    }
      & \text{ otherwise}
  \end{cases}
\]

We then define the strategy \(\tau\) for \(\plI\) in \(G_{fin}(\mc A,\mc B)\)
by \(\tau(s)=\sigma(s^\star)\). Then given any counterattack
\(\alpha\in([\bigcup\mc A]^{<\aleph_0})^\omega\) by \(\plII\) played against 
\(\tau\), we note that \(\alpha^*=\bigcup\{(\alpha\rest n)^*:n<\omega\}\)
is a counterattack to \(\sigma\), and thus loses.
This means \(B=\bigcup\ran{\alpha^*}\not\in\mc B\).

We consider two cases. The first is the case that \(\bigcup\ran{\alpha^*}\)
is finite. Noting that \(\alpha^*(m)\cap\alpha^*(n)=\emptyset\) whenever
\(m\not=n\), there exists \(N<\omega\) such that 
\(\alpha^*(n)=\emptyset\) for all \(n>N\). As a result,
\(\bigcup\ran{\alpha}=\bigcup\ran{\alpha\rest n}\), and thus
\(\bigcup\ran{\alpha}\) is finite, and therefore not in \(\mc B\).

In the other case, \(\bigcup\ran{\alpha^*}\not\in\mc B\) is an 
infinite subset of \(\bigcup\ran{\alpha}\), and therefore 
\(\bigcup\ran{\alpha}\not\in\mc B\) as well. Thus we have shown
that \(\tau\) is a winning strategy for \(\plI\) in
\(G_{fin}(\mc A,\mc B)\).
\end{proof}

We note that the above proof technique could be
used to establish that perfect-information and
limited-information strategies for \(\plII\) in
\(G_{fin}(\mc A,\mc B)\) may be improved to
be valid in \(G_{<2}(\mc A,\mc B)\), provided
\(\mc B\) is \(\Gamma\)-like. As such,
\(G_{<2}(\mc A,\mc B)\) and \(G_{fin}(\mc A,\mc B)\)
are effectively equivalent games under this hypothesis,
so we will no longer consider \(G_{<2}(\mc A,\mc B)\).

\section{Perfect information and predetermined strategies}

We now demonstrate the following, in the spirit of Pawlikowskii's celebrated
result that a winning strategy for the first player in the Rothberger game
may always be improved to a winning predetermined strategy \cite{MR1279482}.

\begin{theorem}\label{pawlikowskii}
Let \(\mc A\) be almost-\(\Gamma\)-like and \(\mc B\) be \(\Gamma\)-like. Then
\begin{itemize}
\item \(\plI\win G_{fin}(\mc A,\mc B)\) if and only if
\(\plI\prewin G_{fin}(\mc A,\mc B)\), and
\item \(\plI\win G_1(\mc A,\mc B)\) if and only if
\(\plI\prewin G_1(\mc A,\mc B)\).
\end{itemize}
\end{theorem}

\begin{proof}
We assume \(\plI\win G_{fin}(\mc A,\mc B)\)
and let the symbol \(\dagger\) mean \(<\aleph_0\)
(respectively, \(\plI\win G_1(\mc A,\mc B)\)
and \(\dagger=1\),
and for convenience we assume \(\plII\) plays
singleton subsets of \(\mc A\) rather than elements).
As \(\mc A\) is almost-\(\Gamma\)-like, there is a 
winning strategy \(\sigma\) where
\(|\sigma(s)|=\aleph_0\) and \(\sigma(s)\cap\bigcup\ran{s}=\emptyset\)
(that is, \(\sigma\) never replays the choices of \(\plII\))
for all partial plays \(s\) by \(\plII\).

For each \(s\in\omega^{<\omega}\), suppose 
\(F_{s\rest m}\in[\bigcup A]^{\dagger}\) 
is defined for each \(0<m\leq|s|\).
Then let \(s^\star:|s|\to[\bigcup\mc A]^{\dagger}\)
be defined by
\(s^\star(m)=F_{s\rest m+1}\), and define \(\tau':\omega^{<\omega}\to\mc A\)
by \(\tau'(s)=\sigma(s^\star)\). Finally, set 
\([\sigma(s^\star)]^{\dagger}=\{F_{s\concat\tuple{n}}:n<\omega\}\), and
for some bijection \(b:\omega^{<\omega}\to\omega\) let \(\tau(n)=\tau'(b(n))\)
be a predetermined strategy for \(\plI\) in \(G_{fin}(\mc A,\mc B)\)
(resp. \(G_1(\mc A,\mc B)\)).

Suppose \(\alpha\) is a counterattack by \(\plII\) against \(\tau\), so 
\[
  \alpha(n)
    \in
  [\tau(n)]^{\dagger}
    =
  [\tau'(b(n))]^{\dagger}
    =
  [\sigma(b(n)^\star)]^{\dagger}
\]
It follows that \(\alpha(n)=F_{b(n)\concat\tuple{m}}\) for some \(m<\omega\).
In particular, there is some infinite subset \(W\subseteq\omega\) and \(f\in\omega^\omega\)
such that \(\{\alpha(n):n\in W\}=\{F_{f\rest n+1}:n<\omega\}\).
Note here that \((f\rest n+1)^\star=(f\rest n)^\star\concat\tuple{F_{f\rest n+1}}\).
This shows that \(F_{f\rest n+1}\in[\sigma((f\rest n)^\star)]^{\dagger}\) 
is an attempt by \(\plII\) to defeat \(\sigma\), which fails. Thus 
\(\bigcup\{F_{f\rest n+1}:n<\omega\}=\bigcup\{\alpha(n):n\in W\}\not\in\mc B\),
and since this set is infinite (as \(\sigma\) prevents \(\plII\)
from repeating choices) we have \(\bigcup\{\alpha(n):n<\omega\}\not\in\mc B\) too.
Therefore \(\tau\) is winning.
\end{proof}

Note that the assumption in Theorem \ref{pawlikowskii} that \(\mc A\)
be almost-\(\Gamma\)-like cannot be omitted. In
\cite{MR3467819}
an example of a space \(X^*\) and point \(\infty\in X^*\) 
where \(\plI\win G_1(\mc A,\mc B)\) but
\(\plI\notprewin G_1(\mc A,\mc B)\) is given, where \(\mc A\) is
the set of open neighborhoods of \(\infty\) 
(which are all uncountable), 
and \(\mc B\) is the set \(\Gamma_{X^*,\infty}\) of sequences converging to that point.
(Note that \(G_1(\mc A,\mc B)\) is called \(Gru_{O,P}(X^*,\infty)\) in that
paper, and an equivalent game \(Gru_{K,P}(X)\) is what is directly
studied. In fact, more is shown: \(\plI\) has a winning perfect-information
strategy, but for any natural number \(k\), any strategy that only uses 
the most recent \(k\) moves of \(\plII\) and the round number
can be defeated.)

While \(\mc A\) is often not almost-\(\Gamma\)-like in general,
it may satisfy that property in combination with the selection principles
being considered.

\begin{proposition}\label{auto-asl}
Let \(\mc B\) be \(\Gamma\)-like, \(\mc B\subseteq\mc A\),
and \(\plI\notprewin G_{fin}(\mc A,\mc B)\). Then
\(\mc A\) is almost-\(\Gamma\)-like.
\end{proposition}
\begin{proof}
Let \(A\in\mc A\), and for all \(n<\omega\) let \(A_n=A\).
Then \(A_n\) is not a winning predetermined strategy for
\(\plI\), so \(\plII\) may choose finite sets
\(B_n\subseteq A_n=A\) such that 
\(A'=\bigcup\{B_n:n<\omega\}\in\mc B\subseteq\mc A\).

It follows that \(A'\subseteq A\) and \(|A'|=\aleph_0\), 
and for any infinite subset
\(A''\subseteq A'\) (in particular, any cofinite subset),
\(A''\in\mc B\subseteq\mc A\). Thus \(\mc A\) is almost-\(\Gamma\)-like.
\end{proof}

Note that in the previous result, 
\(\plI\notprewin G_{fin}(\mc A,\mc B)\) could be weakened
to the choice principle \({\mc A}\choose{\mc B}\): for every
member of \(\mc A\), there is some countable subset belonging to \(\mc B\).

\begin{corollary}
Let \(\mc B\) be \(\Gamma\)-like and \(\mc B\subseteq\mc A\).
Then 
\begin{itemize}
\item \(\plI\win G_{fin}(\mc A,\mc B)\) if and only if
\(\plI\prewin G_{fin}(\mc A,\mc B)\), and
\item \(\plI\win G_{1}(\mc A,\mc B)\) if and only if
\(\plI\prewin G_{1}(\mc A,\mc B)\).
\end{itemize}
\end{corollary}
\begin{proof}
Assuming \(\plI\notprewin G_{fin}(\mc A,\mc B)\), we have
\(\plI\notwin G_{fin}(\mc A,\mc B)\) by Proposition \ref{auto-asl}
and Theorem \ref{pawlikowskii}.

Similarly, assuming \(\plI\notprewin G_{1}(\mc A,\mc B)
\Rightarrow\plI\notprewin G_{fin}(\mc A,\mc B)\), we have
\(\plI\notwin G_{1}(\mc A,\mc B)\) by Proposition \ref{auto-asl}
and Theorem \ref{pawlikowskii}.
\end{proof}

This corollary generalizes e.g. Theorems 26 and 30 of \cite{MR1378387}
Theorem 5 of \cite{MR2119791}, and Corollary 36 of \cite{clontzDualPreprint}.

In summary, using the selection principle notation \(S_\star(\mc A,\mc B)\):

\begin{corollary}
Let \(\mc B\) be \(\Gamma\)-like and \(\mc B\subseteq\mc A\).
Then 
\begin{itemize}
\item \(\plI\notwin G_{fin}(\mc A,\mc B)\) if and only if
\(S_{fin}(\mc A,\mc B)\) if and only if
\(\alpha_2(\mc A,\mc B)\), and
\item \(\plI\notwin G_{1}(\mc A,\mc B)\) if and only if
\(S_{1}(\mc A,\mc B)\) if and only if
\(\alpha_4(\mc A,\mc B)\).
\end{itemize}
\end{corollary}

\section{Disjoint selections}

In each \(\alpha_i(\mc A,\mc B)\) principle, it is not required for the collection
\(\{A_n:n<\omega\}\) to be pairwise disjoint. However, in many cases it may as
well be.

\begin{definition}
For \(i\in\{1,2,3,4\}\) let \(\alpha_{i.1}(\mc A,\mc B)\) denote the claim that
\(\alpha_i(\mc A,\mc B)\) holds provided the collection \(\{A_n:n<\omega\}\)
is pairwise disjoint.
\end{definition}

Of course, \(\alpha_i(\mc A,\mc B)\) implies \(\alpha_{i.1}(\mc A,\mc B)\).
It's also immediate that \(\alpha_{i.1}(\mc A,\mc B)\) implies
\(\alpha_{i.1+1}(\mc A,\mc B)\) for the same reason that \(\alpha_i(\mc A,\mc B)\)
implies \(\alpha_{i+1}(\mc A,\mc B)\). 

We take advantage of the following lemma.

\begin{lemma}[Lemma 1.2 of \cite{MR1195504}]
Given a family \(\{A_n:n<\omega\}\) of infinite sets, there exist infinite subsets
\(A_n'\subseteq A_n\) such that \(\{A_n':n<\omega\}\) is pairwise disjoint.
\end{lemma}

\begin{proposition}\label{pointOne}
Let \(\mc A\) be \(\Gamma\)-like.
For \(i\in\{2,3,4\}\), \(\alpha_i(\mc A,\mc B)\) is equivalent to
\(\alpha_{i.1}(\mc A,\mc B)\).
\end{proposition}

\begin{proof}
Assume \(\alpha_{i.1}(\mc A,\mc B)\).
Let \(A_n\in\mc A\). By applying the previous lemma, we have
\(\{A_n':n<\omega\}\) pairwise disjoint with each \(A_n'\) being
an infinite subset of \(A_n\). Since \(\mc A\) is \(\Gamma\)-like,
\(A_n'\in\mc A\), so we have a witness \(B\in\mc B\) such that
\(A_n'\cap B\) satisfies \(\alpha_{i.1}(\mc A,\mc B)\) for all \(n<\omega\).
Since \(A_n'\subseteq A_n\), it follows that \(A_n\cap B\) satisfies
\(\alpha_{i}(\mc A,\mc B)\) for all \(n<\omega\).
\end{proof}

It's also true that \(\alpha_1(\Gamma_{X,x},\Gamma_{X,x})\)
is equivalent to \(\alpha_{1.1}(\Gamma_{X,x},\Gamma_{X,x})\),
which is captured by the following theorem.

\begin{theorem}
Let \(\mc A\) be a \(\Gamma\)-like collection closed under finite unions
and \(\mc A\subseteq\mc B\).
Then \(\alpha_1(\mc A,\mc B)\) is equivalent to
\(\alpha_{1.1}(\mc A,\mc B)\).
\end{theorem}
 
\begin{proof}
Let \(A_n\in\mc A\) and assume \(\alpha_{1.1}(\mc A,\mc B)\).
To apply the assumption, we will define a pairwise disjoint
collection \(\{A_n':n<\omega\}\). First let \(0'=0\) and \(A_0'=A_0\).
Then suppose \(m'\geq m\) and \(A_m'\subseteq A_{m'}\subseteq\bigcup_{i\leq m}A_i'\) 
are defined for all \(m\leq n\).

If \(A_k\setminus\bigcup_{m\leq n}A_m'\) is finite for \(k>n'\), let
\(B=\bigcup_{m\leq n'}A_m\in\mc A\subseteq\mc B\). This \(B\) then witnesses 
\(\alpha_1(\mc A,\mc B)\) since \(A_k\setminus B\) is finite for all \(k<\omega\).

Otherwise pick the minimal \((n+1)'>n\) where 
\(A_{n+1}'=A_{(n+1)'}\setminus\bigcup_{m\leq n}A_m'\) is infinite.
It follows that \(A_{n+1}'\subseteq A_{(n+1)'}\subseteq \bigcup_{m\leq n+1}A_m'\).
By construction, \(\{A_n':n<\omega\}\) is a pairwise disjoint collection of
members of \(\mc A\), and we may
apply \(\alpha_{1.1}(\mc A,\mc B)\) to obtain \(B\in\mc B\) where
\(A_n'\setminus B\) is finite for all \(n<\omega\).

Finally let \(k<\omega\). If \(k=n'\) for some \(n<\omega\), then 
\(A_k\setminus B=A_{n'}\setminus B\subseteq (\bigcup_{m\leq n}A_m')\setminus B\) 
is finite.
Otherwise, \(n'<k<(n+1)'\) for some \(n<\omega\).
Then 
\((A_k\setminus\bigcup_{m\leq n}A_m')\setminus B\subseteq A_k\setminus\bigcup_{m\leq n}A_m'\)
is finite, and 
\((A_k\cap\bigcup_{m\leq n}A_m')\setminus B\subseteq(\bigcup_{m\leq n}A_m')\setminus B\)
is finite, showing \(A_k\setminus B\) is finite.
\end{proof}

Another fractional version of these \(\alpha\)-principles is given as
\(\alpha_{1.5}\) in \cite{MR1195504}, defined in general as follows.

\begin{definition}
Let \(\alpha_{1.5}(\mc A,\mc B)\) be the assertion that when \(A_n\in\mc A\)
and \(\{A_n:n<\omega\}\) is pairwise disjoint, then there exists \(B\in\mc B\)
such that \(A_n\cap B\) is cofinite in \(A_n\) for infinitely-many \(n<\omega\).
\end{definition}

It's immediate from their definitions that 
\(\alpha_{1.1}(\mc A,\mc B)\) implies \(\alpha_{1.5}(\mc A,\mc B)\), which
implies \(\alpha_{3.1}(\mc A,\mc B)\).
Nyikos originally showed that \(\alpha_{1.5}(\Gamma_{X,x},\Gamma_{X,x})\) implies
\(\alpha_2(\Gamma_{X,x},\Gamma_{X,x})\); this result generalizes as follows.

\begin{theorem}
Let \(\mc A\) be a \(\Gamma\)-like collection closed under finite unions.
Then \(\alpha_{1.5}(\mc A,\mc B)\) implies
\(\alpha_{2}(\mc A,\mc B)\).
\end{theorem}

\begin{proof}
We assume \(\alpha_{1.5}(\mc A,\mc B)\) and demonstrate
\(\alpha_{2.1}(\mc A,\mc B)\), which is equivalent to
\(\alpha_{2}(\mc A,\mc B)\) by Proposition \ref{pointOne}.
So let \(A_n\in\mc A\) such that \(\{A_n:n<\omega\}\) is pairwise-disjoint.

We may partition each \(A_n\) into \(\{A_{n,m}:m<\omega\}\) with
\(A_{n,m}\in\mc A\) for all \(m<\omega\). 
Let \(A_n'=\bigcup\{A_{i,j}:i+j=n\}\in\mc A\);
since \(\{A_n':n<\omega\}\) is pairwise disjoint, we may apply
\(\alpha_{1.5}(\mc A,\mc B)\) to obtain \(B\in\mc B\) where
\(A_n'\cap B\) is cofinite in \(A_n'\) for infinitely-many \(n<\omega\).

Then for \(n<\omega\), choose \(N\geq n\) with \(A_N'\cap B\) cofinite in \(A_N'\).
Then \(A_{n,N-n}\subseteq A_N'\), 
so \(A_{n,N-n}\cap B\) is cofinite in \(A_{n,N-n}\), in particular,
\(A_{n,N-n}\cap B\) is infinite.
Therefore \(A_n\cap B\) is infinite, and we have shown
\(\alpha_{2.1}(\mc A,\mc B)\).
\end{proof}

\begin{corollary}
Let \(\mc A\) be a \(\Gamma\)-like collection closed under finite unions.
Then \(\alpha_x(\mc A,\mc B)\) implies \(\alpha_y(\mc A,\mc B)\) for
\(1<x\leq y\). Additionally, if \(\mc A\subseteq\mc B\), then 
\(\alpha_x(\mc A,\mc B)\) implies \(\alpha_y(\mc A,\mc B)\) for
\(1\leq x\leq y\).
\end{corollary}

For this paragraph we adopt the conventional assumption that
\(\Gamma_{X,x}\) is restricted to countable sets.
Nyikos showed a consistent example
where \(\alpha_2(\Gamma_{X,x},\Gamma_{X,x})\) fails to imply
\(\alpha_{1.5}(\Gamma_{X,x},\Gamma_{X,x})\), and a consistent example where
\(\alpha_{1.5}(\Gamma_{X,x},\Gamma_{X,x})\) fails to imply
\(\alpha_{1}(\Gamma_{X,x},\Gamma_{X,x})\) \cite{MR1195504}.
On the other hand, Dow showed that \(\alpha_2(\Gamma_{X,x},\Gamma_{X,x})\)
implies \(\alpha_{1}(\Gamma_{X,x},\Gamma_{X,x})\) in the Laver model
for the Borel conjecture \cite{MR975638}; the author conjectures
that this model (specifically, the fact that
every \(\omega\)-splitting family contains an \(\omega\)-splitting
family of size less than \(\mathfrak b\) in this model) witnesses
an affirmative answer to the following question.

\begin{definition}
A \(\Gamma\)-like collection is \term{strongly-\(\Gamma\)-like}
if the collection is closed under finite unions and each member
is countable.
\end{definition}

\begin{question}
Let \(\mc A\) be strongly-\(\Gamma\)-like.
Is it consistent that \(\alpha_2(\mc A,\mc A)\)
implies \(\alpha_1(\mc A,\mc A)\)?
\end{question}

\section{Conclusion}

We conclude with the following easy result, and a couple questions.

\begin{proposition}
Let \(\mc B\) be \(\Gamma\)-like. Then \(\alpha_1(\mc A,\mc B)\) holds if and only
if \(\plI\notprewin G_{cf}(\mc A,\mc B)\).
\end{proposition}

\begin{proof}
We first assume \(\alpha_1(\mc A,\mc B)\) and let \(A_n\in\mc A\) for \(n<\omega\)
define a predetermined strategy for \(\plI\). By \(\alpha_1(\mc A,\mc B)\), we
immediately obtain \(B\in\mc B\) such that \(|A_n\setminus B|<\aleph_0\). Thus
\(B_n=A_n\cap B\) is a cofinite choice from \(A_n\), and 
\(B'=\bigcup\{B_n:n<\omega\}\) is an infinite subset of \(B\),
so \(B'\in\mc B\). Thus \(\plII\) may defeat \(\plI\) by choosing
\(B_n\subseteq A_n\) each round, witnessing \(\plI\notprewin G_{cf}(\mc A,\mc B)\).

On the other hand, let \(\plI\notprewin G_{cf}(\mc A,\mc B)\). Given \(A_n\in\mc A\)
for \(n<\omega\), we note that \(\plII\) may choose a cofinite subset \(B_n\subseteq A_n\)
such that \(B=\bigcup\{B_n:n<\omega\}\in\mc B\). Then \(B\) witnesses \(\alpha_1(\mc A,\mc B)\)
since \(|A_n\setminus B|\leq|A_n\setminus B_n|\leq\aleph_0\).
\end{proof}

\begin{question}
Is there a game-theoretic characterization of \(\alpha_3(\mc A,\mc B)\)?
\end{question}

Noting that \(\plI\win G_1(\Gamma_X,\Gamma_X)\) if and only if
\(\plI\win G_{fin}(\Gamma_X,\Gamma_X)\) \cite{MR2417134}, but
the same is not true of \(G_\star(\Gamma_{X,x},\Gamma_{X,x})\)
(i.e. there are \(\alpha_4\) spaces that are not \(\alpha_2\)
\cite{shakhmatov2002convergence}),
we also ask the following.

\begin{question}
Is there a natural condition on \(\mc A,\mc B\) guaranteeing
\(\plI\win G_1(\mc A,\mc B)\Rightarrow\plI\win G_{fin}(\mc A,\mc B)\)?
\end{question}

\section{Acknowledgements}

The author would like to thank Alan Dow, Jared Holshouser,
and Alexander Osipov for various discussions related to this paper.

\bibliographystyle{plain}
\bibliography{bibliography}

\end{document}